\documentclass[12pt,reqno]{amsart}

\usepackage{amssymb}

\usepackage[all]{xy}

\numberwithin{equation}{section}

\newtheorem{theorem}{Theorem}[section]
\newtheorem{proposition}[theorem]{Proposition}
\newtheorem{lemma}[theorem]{Lemma}
\newtheorem{corollary}[theorem]{Corollary}

\begin{document}

\title[Semiampleness and \'etale triviality of vector bundles]{Semiample vector bundles and a
characterization of \'etale trivial vector bundles}

\author[I. Biswas]{Indranil Biswas}

\address{Department of Mathematics, Shiv Nadar University, NH91, Tehsil Dadri,
Greater Noida, Uttar Pradesh 201314, India}

\email{indranil.biswas@snu.edu.in, indranil29@gmail.com}

\author[D. S. Nagaraj]{D. S. Nagaraj}

\address{Indian Institute of science education and research, Tirupati, Srinivasapuram-Jangalapalli Village, 
Panguru (G.P) Yerpedu Mandal, Tirupati - 517619, Chittoor Dist., Andhra Pradesh, India}

\email{dsn@labs.iisertirupati.ac.in}



\date{}

\begin{abstract}
We prove that a vector bundle $E$ over a smooth complex projective variety $M$ is \'etale trivial if and
only if $E$ is semiample and $c_1(E)\, \in\, H^2(M,\, {\mathbb Q})$ vanishes. Also, a vector bundle $E$
over a smooth complex projective curve is semiample if and only if it is virtually globally generated.
\end{abstract}

\maketitle

\section{Introduction}

Consider a vector bundle $E$ over an irreducible smooth complex projective variety $M$. It is called
\'etale trivial if its pullback to some finite \'etale covering of $M$ is algebraically trivial.
So $E$ is \'etale trivial if and only if it admits an integrable algebraic connection for which the
image of the monodromy homomorphism is a finite group.
A theorem of Nori says that $E$ is \'etale trivial if and only if it satisfies some nontrivial
polynomial equation whose coefficients are integers (this is explained in Section \ref{se2.1}).

Semiample vector bundles lie between the ample vector bundles and the nef vector bundles. A line bundle
is called semiample if some positive tensor power of it is generated by its global sections. A vector bundle
is semiample if the corresponding tautological line bundle is semiample.

We prove the following (see Theorem \ref{thm1}):

\textit{Let $E$ be a vector bundle over an irreducible smooth complex projective variety $M$. Then
the following two statements are equivalent:}
\begin{enumerate}
\item \textit{$E$ is \'etale trivial.}

\item \textit{$E$ is semiample and $c_1(E)\,=\, 0$.}
\end{enumerate}

Let
$$
0\, \longrightarrow\, V \, \longrightarrow\, E \, \longrightarrow\,
Q \, \longrightarrow\, 0
$$
be a short exact sequence of vector bundles on $M$. We prove the following (see Proposition
\ref{prop2}):

\textit{If $E$ is semiample and $c_1(Q)\,=\, 0$, then the above exact sequence splits.}

In the final section we consider semiample vector bundles on curves. The following is proved
(see Theorem \ref{thm2}):

\textit{A vector bundle $E$ on a smooth complex projective curve $M$ is semiample if and only if $E$
is virtually globally generated.}

\section{Semiampleness and finite bundles}

\subsection{Definitions}\label{se2.1}

Let $M$ be an irreducible complex projective variety. A line bundle $L$ on $M$ is called
\textit{semiample} if $L^{\otimes n}$ is generated by its global sections for some $n\, \geq\,1$
\cite[p.~354, Definition 1.1]{Fu}.
A vector bundle $E$ on $M$ is called \textit{semiample} if the tautological line bundle
${\mathcal O}_{\mathbb{P}(E)}(1)\, \longrightarrow\, \mathbb{P}(E)$ is semiample, where
$\mathbb{P}(E)$ is the projective bundle parametrizing the one-dimensional quotients of the
fibers of $E$.

A line bundle $L$ on $M$ is called
\textit{nef} if the restriction of $L$ to every closed curve on $M$ is of nonnegative degree.
A vector bundle $E$ on $M$ is called \textit{nef} if the tautological line bundle
${\mathcal O}_{\mathbb{P}(E)}(1)$ is nef. Every semiample vector bundle is nef. Any quotient bundle of
a semiample (respectively, nef) vector bundle is semiample (respectively, nef). See \cite{La1},
\cite{La2} for nef bundles and semiample bundles.

Let $E$ be a vector bundle over $M$. By $E^{\otimes 0}$ we denote the trivial line bundle over $M$. Also,
$E^{\oplus 0}$ denotes the vector bundle of rank zero.
For any polynomial $p(t)\,:=\, \sum_{i=0}^n c_i t^i$, where $c_i$ are nonnegative integers, define
$$
p(E)\ :=\ \bigoplus_{i=0}^n (E^{\otimes i})^{\oplus c_i}.
$$
A vector bundle $E$ is called \textit{finite} if there are two distinct polynomials $p(t)\, \not=\, q(t)$, whose
coefficients are nonnegative integers, such that the vector bundles $p(E)$ and $q(E)$ are isomorphic
\cite{We}, \cite{No1}, \cite{No2}.

A vector bundle $E$ on $M$ is called \textit{\'etale trivial} if there is a finite \'etale covering map
$$
\varpi\ :\ \widetilde{M} \ \longrightarrow\ M
$$
such that the pullback $\varpi^*E$ is an algebraically trivial vector bundle. Note that when $M$ is smooth, a
vector bundle $E$ on $M$ is \'etale trivial if and only if $E$ admits an integrable algebraic
connection whose monodromy homomorphism $\pi_1(M,\, x_0)\, \longrightarrow\, \text{GL}(E_{x_0})$ has
the property that its image is a finite subgroup of $\text{GL}(E_{x_0})$.

A special case of a theorem of Nori says that $E$ is finite if and only if it is \'etale trivial
\cite{No1}, \cite{No2}.

For a vector bundle $E$ on a smooth projective variety $M$, by $c_i(E)$ we denote the
Chern class of $E$ in $H^{2i}(M,\, {\mathbb Q})$.

\subsection{A characterization of finite bundles}

Let $M$ be a connected smooth complex projective variety. When $\dim M\, >\, 1$, fix a polarization (same as
the first Chern class of an ample line bundle) on $M$ in order to define the degree of torsion-free coherent
sheaves on $M$. This allows us to define semistable and stable vector bundles on $M$. We note that for a
vector bundle $E$ with $c_1(E)\,=\, 0\, =\, c_2(E)$, the semistability condition for it is independent of the 
choice of the polarization on $M$. Indeed, a vector bundle $E$ with $c_1(E)\,=\, 0\, =\, c_2(E)$ is semistable
if and only if the pullback of $E$ by every map from smooth projective curves to $M$ is semistable
\cite[p.~3--4, Theorem 1.2]{BB}, \cite{Si}.

\begin{theorem}\label{thm1}
Let $E$ be a vector bundle over $M$. Then the following two statements are equivalent:
\begin{enumerate}
\item $E$ is \'etale trivial.

\item $E$ is semiample and $c_1(E)\,=\, 0$.
\end{enumerate}
\end{theorem}

\begin{proof}
First assume that $E$ is \'etale trivial. Since $E$ is \'etale trivial, by definition, there is
a finite \'etale covering map
$$
\varpi\,:\, Y\, \longrightarrow\, M
$$
such that the pullback $\varpi^*E$ is a trivial vector bundle. Replacing $Y$ by a further
\'etale covering of $Y$, if necessary, we may assume that the map $\varpi$ is actually Galois. Let
$$
\Gamma\,\,=\,\, \text{Gal}(Y/M) \,\, \subset\, \, \text{Aut}(Y)
$$
be the Galois group of $\varpi$. The action of the Galois group $\Gamma$ on $Y$ has a natural lift
to an action of $\Gamma$ on the pulled back vector bundle
$\varpi^* E$. This action of $\Gamma$ on $\varpi^*E$ produces actions of $\Gamma$
on both ${\mathbb P}(\varpi^* E)$ and ${\mathcal O}_{{\mathbb P}(\varpi^* E)}(1)$. The natural projections
${\mathcal O}_{{\mathbb P}(\varpi^* E)}(1)\, \longrightarrow\, {\mathbb P}(\varpi^* E)$ and
${\mathbb P}(\varpi^* E)\, \longrightarrow\, Y$ are evidently $\Gamma$--equivariant maps.
Let
\begin{equation}\label{e1}
\beta\, :\, {\mathbb P}(\varpi^* E) \,\longrightarrow\, {\mathbb P}(\varpi^* E)/\Gamma
\,=\, {\mathbb P}(E)
\end{equation}
be the natural map. Note that we have $\beta^* {\mathcal O}_{{\mathbb P}(E)}(1)\,=\,
{\mathcal O}_{{\mathbb P}(\varpi^* E)}(1)$.

The action of $\Gamma$ on ${\mathcal O}_{{\mathbb P}(\varpi^* E)}(1)$ produces an action of
$\Gamma$ on ${\mathcal O}_{{\mathbb P}(\varpi^* E)}(\lambda)\,=\,
{\mathcal O}_{{\mathbb P}(\varpi^* E)}(1)^{\otimes \lambda}$ for all $\lambda\, \in\, {\mathbb Z}$.
The action of $\Gamma$ on ${\mathcal O}_{{\mathbb P}(\varpi^* E)}(\lambda)$ produces an action of
$\Gamma$ on $$H^0({\mathbb P}(\varpi^* E),\, {\mathcal O}_{{\mathbb P}(\varpi^* E)}(\lambda)).$$ Let
\begin{equation}\label{el}
H^0({\mathbb P}(\varpi^* E),\, {\mathcal O}_{{\mathbb P}(\varpi^* E)}(\lambda))^\Gamma\,
\, \subset\,\, H^0({\mathbb P}(\varpi^* E),\, {\mathcal O}_{{\mathbb P}(\varpi^* E)}(\lambda))
\end{equation}
be the subspace of $H^0({\mathbb P}(\varpi^* E),\, {\mathcal O}_{{\mathbb P}(\varpi^* E)}(\lambda))$ on
which the finite group $\Gamma$ acts trivially.

We will show that there is an integer $\delta \, \geq\, 1$ such that the subspace
$$
H^0({\mathbb P}(\varpi^* E),\, {\mathcal O}_{{\mathbb P}(\varpi^* E)}(\delta))^\Gamma\,
\, \subset\,\, H^0({\mathbb P}(\varpi^* E),\, {\mathcal O}_{{\mathbb P}(\varpi^* E)}(\delta))
$$
(see \eqref{el})
is base point free, meaning for every point $z\, \in\, {\mathbb P}(\varpi^* E)$ there is a
$\Gamma$--invariant section
$$s_z\, \in\, H^0({\mathbb P}(\varpi^* E),\, {\mathcal O}_{{\mathbb P}(\varpi^* E)}(\delta))^\Gamma$$
such that $s_z(z)\, \not=\, 0$.

To show this, first consider the trivial algebraic vector bundle
$${\mathcal H}\ :=\ Y\times H^0(Y,\, \varpi^* E)\ \longrightarrow\ Y$$ over $Y$ with fiber
$H^0(Y,\, \varpi^* E)$. Let
\begin{equation}\label{ei}
{\mathcal I}\ :\ {\mathcal H}\ \longrightarrow\ \varpi^* E
\end{equation}
be the evaluation map that sends any $(y,\,\omega)\, \in\, Y\times H^0(Y,\, \varpi^* E)$ to
$\omega(y)\, \in\, (\varpi^* E)_y$. Since $\mathcal H$ is a
algebraically trivial vector bundle, the map ${\mathcal I}$
in \eqref{ei} is actually an isomorphism.

The action of $\Gamma$ on $\varpi^* E$ produces an action of $\Gamma$ on $H^0(Y,\, \varpi^* E)$. This
action of $\Gamma$ on $H^0(Y,\, \varpi^* E)$ and the Galois action of $\Gamma$ on $Y$
together produce a diagonal action of $\Gamma$ on ${\mathcal H}\, =\, Y\times H^0(Y,\, \varpi^* E)$.
Now it is evident that the map ${\mathcal I}$ in \eqref{ei} is $\Gamma$--equivariant. Let
\begin{equation}\label{ei2}
\widetilde{\mathcal I}\ :\ {\mathbb P}({\mathcal H})\ =\
Y\times {\mathbb P}(H^0(Y,\, \varpi^* E)) \ \longrightarrow\ {\mathbb P}(\varpi^* E)
\end{equation}
be the isomorphism of projective bundles induced by the isomorphism $\mathcal I$ of vector bundles
in \eqref{ei}.
Since $\mathcal I$ is $\Gamma$--equivariant, it follows that the isomorphism $\widetilde{\mathcal I}$
in \eqref{ei2} is also $\Gamma$--equivariant. So $\widetilde{\mathcal I}$ produces an isomorphism
\begin{equation}\label{ei3}
\widetilde{\mathcal I}'\ :\ {\mathbb P}({\mathcal H})/\Gamma\ =\
(Y\times {\mathbb P}(H^0(Y,\, \varpi^* E)))/\Gamma \ \longrightarrow\ {\mathbb P}(\varpi^* E)/\Gamma\
=\ {\mathbb P}(E).
\end{equation}
On the other hand, there is a natural projection
\begin{equation}\label{ei4}
{\mathbf q}\ : \ {\mathbb P}({\mathcal H})/\Gamma\ =\
(Y\times {\mathbb P}(H^0(Y,\, \varpi^* E)))/\Gamma \ \longrightarrow\ 
{\mathbb P}(H^0(Y,\, \varpi^* E))/\Gamma
\end{equation}
that sends the $\Gamma$-orbit of $(y,\, \omega)\, \in\, Y\times {\mathbb P}(H^0(Y,\, \varpi^* E))$
to the $\Gamma$-orbit of $\omega$. Combining \eqref{ei3} and \eqref{ei4}, we have the projection
\begin{equation}\label{ei5}
\widehat{\mathcal I} \ :=\ {\mathbf q}\circ (\widetilde{\mathcal I}')^{-1}\ :\
{\mathbb P}(E) \ \longrightarrow\ {\mathbb P}(H^0(Y,\, \varpi^* E))/\Gamma .
\end{equation}

Let
\begin{equation}\label{ei6}
\textbf{Q}\ :\ {\mathbb P}(H^0(Y,\, \varpi^* E))\ \longrightarrow \ {\mathbb P}(H^0(Y,\, \varpi^* E))/\Gamma
\end{equation}
be the quotient map for the action of $\Gamma$. Consider the
action of $\Gamma$ on ${\mathbb P}(H^0(Y,\, \varpi^* E))$, and construct the line bundle
\begin{equation}\label{ei7}
{\mathcal L}_0 \ :=\ \bigotimes_{\gamma\in \Gamma} \gamma^* {\mathcal O}_{{\mathbb P}(H^0(Y,\, \varpi^* E))}(1)
\ \longrightarrow\ {\mathbb P}(H^0(Y,\, \varpi^* E)).
\end{equation}
Note that ${\mathcal L}_0$ is isomorphic to ${\mathcal O}_{{\mathbb P}(H^0(Y,\, \varpi^* E))}(d_0)$,
where $d_0\,=\, \# \Gamma$ is the cardinality of the group $\Gamma$. From the construction of ${\mathcal L}_0$
it is evident that there is a unique line bundle
${\mathcal L}$ on ${\mathbb P}(H^0(Y,\, \varpi^* E))/\Gamma$ such that we have
\begin{equation}\label{ei8}
\textbf{Q}^*{\mathcal L}\ =\ {\mathcal L}_0
\end{equation}
as $\Gamma$--equivariant line bundles,
where $\textbf{Q}$ and ${\mathcal L}_0$ are constructed in \eqref{ei6} and \eqref{ei7} respectively. Since
the map $\textbf{Q}$ is surjective and ${\mathcal L}_0$ is ample, from \eqref{ei8} it follows immediately
that ${\mathcal L}$ is ample. Therefore, there is a positive integer $n_0$ such that the line bundle
\begin{equation}\label{ei9}
{\mathcal L}_1\ := \ {\mathcal L}^{\otimes n_0}
\end{equation}
is generated by its global sections. Now consider the line bundle
\begin{equation}\label{ei10}
\widetilde{\mathcal L}\ := \ \widehat{\mathcal I}^*{\mathcal L}_1\ \longrightarrow\
{\mathbb P}(E),
\end{equation}
where $\widehat{\mathcal I}$ and ${\mathcal L}_1$ are constructed in \eqref{ei5} and
\eqref{ei9} respectively. Since ${\mathcal L}_1$ is generated by its global sections, we know
that the line bundle $\widetilde{\mathcal L}$ in \eqref{ei10} is also generated by its global sections.

{}From \eqref{ei7} and \eqref{ei9} it follows that $\beta^*\widetilde{\mathcal L}\,=\,
{\mathcal O}_{{\mathbb P}(\varpi^* E)}(n_0d_0)$, where $\beta$ is the map in \eqref{e1}.
T explain this, note that the
map $$\widehat{\mathcal I}\circ\beta\ :\ {\mathbb P}(\varpi^* E) \
\longrightarrow\, {\mathbb P}(H^0(Y,\, \varpi^* E))/\Gamma$$ coincides with the composition of maps
$$
{\mathbb P}(\varpi^* E) \, \xrightarrow{\,\,\, \widetilde{\mathcal I}^{-1}\,\,\,}\,
Y\times {\mathbb P}(H^0(Y,\, \varpi^* E)) \, \longrightarrow\, {\mathbb P}(H^0(Y,\, \varpi^* E))
\, \stackrel{\mathbf Q}{\longrightarrow}\, {\mathbb P}(H^0(Y,\, \varpi^* E))/\Gamma,
$$
where $Y\times {\mathbb P}(H^0(Y,\, \varpi^* E)) \, \longrightarrow\, {\mathbb P}(H^0(Y,\, \varpi^* E))$
is the natural projection to the second factor. Therefore,
from \eqref{ei7} and \eqref{ei9} it follows that $\beta^*\widetilde{\mathcal L}\,=\,
{\mathcal O}_{{\mathbb P}(\varpi^* E)}(n_0d_0)$. Next note that
\begin{equation}\label{ei11}
H^0({\mathbb P}(E),\, \widetilde{\mathcal L})\ =\ H^0({\mathbb P}(\varpi^* E),\,
\beta^*\widetilde{\mathcal L})^\Gamma\ = \ H^0({\mathbb P}(\varpi^* E),\,
{\mathcal O}_{{\mathbb P}(\varpi^* E)}(n_0d_0))^\Gamma.
\end{equation}
It was observed above that the line bundle $\widetilde{\mathcal L}$ is generated by its global sections.
Consequently, from \eqref{ei11} we conclude that the line
bundle ${\mathcal O}_{{\mathbb P}(\varpi^* E)}(n_0d_0)$ is generated
by $H^0({\mathbb P}(\varpi^* E),\,
{\mathcal O}_{{\mathbb P}(\varpi^* E)}(n_0d_0))^\Gamma$. This proves the statement that
there is an integer $\delta \, \geq\, 1$ such that the subspace
$$
H^0({\mathbb P}(\varpi^* E),\, {\mathcal O}_{{\mathbb P}(\varpi^* E)}(\delta))^\Gamma\,
\, \subset\,\, H^0({\mathbb P}(\varpi^* E),\, {\mathcal O}_{{\mathbb P}(\varpi^* E)}(\delta))
$$
is base point free.

Fix an integer $\delta \, \geq\, 1$ such that the subspace
\begin{equation}\label{a1}
H^0({\mathbb P}(\varpi^* E),\, {\mathcal O}_{{\mathbb P}(\varpi^* E)}(\delta))^\Gamma\,
\, \subset\,\, H^0({\mathbb P}(\varpi^* E),\, {\mathcal O}_{{\mathbb P}(\varpi^* E)}(\delta))
\end{equation}
is base point free.
Since $\beta^* {\mathcal O}_{{\mathbb P}(E)}(1)\,=\, {\mathcal O}_{{\mathbb P}(\varpi^* E)}(1)$, the
subspace in \eqref{a1} corresponds to
$H^0({\mathbb P}(E),\, {\mathcal O}_{{\mathbb P}(E)}(\delta))$.
Given that the subspace in \eqref{a1} is base point free, we conclude that
${\mathcal O}_{{\mathbb P}(E)}(\delta)$ is base point free as well.
Therefore, $E$ is semiample.

Since $\varpi^*E$ is a trivial vector bundle, we have $c_1(\varpi^* E)\,=\, 0$.
As $\varpi$ is a finite covering map, this implies that $c_1(E)\,=\, 0$.
Therefore, the second statement in the theorem holds.

To prove the converse, assume that $E$ is semiample and $c_1(E)\,=\, 0$.
The vector bundle $E$ is nef because it is semiample. This implies that
the exterior product $\bigwedge^{r-1}E$ is nef, where
$r\,=\, {\rm rank}(E)$ \cite[p.~307, Proposition 1.14]{DPS}. Since $c_1(E)\,=\, 0$, the line bundle
$\bigwedge^{r}E^*$ is nef. Consequently, the vector bundle
$$
E^*\,\,=\,\, (\bigwedge\nolimits ^{r-1}E)\otimes (\bigwedge\nolimits ^{r}E^*)
$$
is nef. Thus the vector bundle $E$ is numerically flat \cite[p.~311, Definition 1.17]{DPS}. This implies that
$E$ is semistable \cite[p.~311, Theorem 1.18]{DPS}. Also, we have $c_i(E)\,=\,0$ for all $i\, \geq\, 1$
\cite[p.~311, Corollary 1.19]{DPS}.

Any semistable vector bundle $V$, defined over a complex irreducible smooth projective variety, with
$c_1(V)\,=\,0\,=\, c_2(V)$ has a canonical integrable algebraic connection \cite[p.~20, Corollary 1.3]{Si},
\cite[p.~40, Corollary 3.10]{Si} (this Corollary 3.10 shows that any semistable vector bundle $V$
on a smooth projective variety with $c_1(V)\,=\,0\,=\, c_2(V)$ satisfies the condition in
Corollary 1.3). This canonical connection has the following properties:
\begin{enumerate}
\item The canonical connection on a trivial vector bundle is the trivial connection.

\item Canonical connections on $V\oplus W$ and $V\otimes W$ coincide with the connections
induced by the canonical connections on $V$ and $W$.

\item The canonical connection on $V^*$ coincides with the one induced the canonical connection on $V$.

\item When $W$ and $V$ have canonical connections, every homomorphism of coherent sheaves $V\, \longrightarrow
\, W$ is flat with respect to the canonical connections on $V$ and $W$.

\item If $f\, :\, M_1\, \longrightarrow\, M_2$ is a morphism between smooth complex projective varieties,
and $\nabla$ is the canonical connection on $V\, \longrightarrow\, M_2$, then $f^*V$ is semistable with
$c_1(f^*V)\,=\, 0\,=\, c_2(f^*V)$, and the canonical connection on $f^*V$ coincides with $f^*\nabla$.
\end{enumerate}
(The last property will not be needed here.)

Fix a positive integer $d$ such that ${\mathcal O}_{{\mathbb P}(E)}(d)$ is base point free.
(Such an integer exists because $E$ is semiample.)

Let $\nabla^0$ denote the canonical connection on $E$. The canonical connection on
$\text{Sym}^d(E)$ will be denoted by $\nabla$; so $\nabla$ is induced by $\nabla^0$. Let
$$
{\mathcal V}\,:=\, M\times H^0(M,\, \text{Sym}^d(E))\,\longrightarrow\, M
$$
be the trivial vector bundle over $M$ with fiber $H^0(M,\, \text{Sym}^d(E))$. Consider the evaluation map
\begin{equation}\label{e4}
\Phi\,\, :\,\, {\mathcal V} \,\, \longrightarrow\,\, \text{Sym}^d(E)
\end{equation}
that sends any $(x,\, v)\, \in\, M\times H^0(M,\, \text{Sym}^d(E))$ to $v(x)\, \in\,
\text{Sym}^d(E)_x$. Any section of $\text{Sym}^d(E)$ vanishing at a point actually vanishes identically.
Indeed, this is a consequence of the above first and fourth properties of the canonical connection; any
flat section vanishing at a point must vanish identically.
Hence ${\mathcal V}$ is a subbundle of $\text{Sym}^d(E)$ via $\Phi$ in \eqref{e4}.

The canonical connection $\nabla$ on $\text{Sym}^d(E)$ preserves the subbundle $\mathcal V$ and it
induces the trivial connection on $\mathcal V$. Fix a point $x_0\, \in\, M$, and consider the monodromy
representation
$$
\rho\, :\, \pi_1(M,\, x_0)\, \longrightarrow\, \text{GL}(E_{x_0}).
$$
{}From Lemma \ref{lem1} we conclude that the image of $\rho$ is a finite subgroup of $\text{GL}(E_{x_0})$.
This implies that $E$ is \'etale trivial.
\end{proof}

\subsection{A lemma on symmetric products}

Let $V$ be a finite dimensional complex vector space. Fix a positive integer $d$, and consider the
natural action of $\text{GL}(V)$ on $\text{Sym}^d(V)$ given by the standard action of $\text{GL}(V)$
on $V$. Let
$$
W\, \, \subset\,\, \text{Sym}^d(V)
$$
be a linear subspace such that the partial linear system $P(W)\, \subset\, P(\text{Sym}^d(V))\,=\,
\big\vert{\mathcal O}_{{\mathbb P}(V)}(d)\big\vert$ on ${\mathbb P}(V)$ is base point free; this
means that for every $z\, \in\, {\mathbb P}(V)$ there is a $s_z\, \in\, W$ such that
$s_z(z) \,\in\, {\mathcal O}_{{\mathbb P}(V)}(d)_z$ is nonzero. Let
$$
H\,\, \subset\,\, \text{GL}(V)
$$
be a subgroup that fixes $W$ pointwise.

\begin{lemma}\label{lem1}
The subgroup $H$ is a finite group.
\end{lemma}

\begin{proof}
Let
$$
\varphi\, :\, {\mathbb P}(V)\, \longrightarrow\, {\mathbb P}(W)
$$
be the natural map that sends any $x\, \in\, {\mathbb P}(V)$ to the hyperplane in $W$ given by
all sections that vanish at $x$. This map is well-defined because the partial linear system
$P(W)$ is base point free. Note that $\varphi$ is a finite morphism because
$\varphi^* {\mathcal O}_{{\mathbb P}(W)}(1)\,=\, {\mathcal O}_{{\mathbb P}(V)}(d)$.

Let $$\widetilde{H}\, \subset\, \text{PGL}(V)$$ be the image of
$H$ under the quotient map
$$
\text{GL}(V)\,\longrightarrow\, \text{GL}(V)/{\mathbb C}^*\,=:\, \text{PGL}(V).
$$
The map $\varphi$ is evidently equivariant for the actions of $\widetilde H$ on ${\mathbb P}(V)$
and ${\mathbb P}(W)$. Since $\widetilde H$ acts trivially on ${\mathbb P}(W)$, and $\varphi$ is
a finite morphism, the group $\widetilde H$ is finite (note that the action of
$\text{PGL}(V)$ on ${\mathbb P}(V)$ is faithful).

Next note that
$$
H\cap ({\mathbb C}\setminus\{0\})\,\,\subset\,\, \text{GL}(V)
$$
is a finite group because it fixes some nonzero element of $\text{Sym}^d(V)$. Thus
$H$ is an extension of the finite group $\widetilde H$ by the finite group $H\cap ({\mathbb C}\setminus\{0\})$.
Consequently, $H$ is a finite group.
\end{proof}

\section{Splitting of \'etale trivial quotients}

As before, $M$ is a connected smooth complex projective variety.
Let
\begin{equation}\label{e12}
0\, \longrightarrow\, V \, \longrightarrow\, E \, \stackrel{p}{\longrightarrow}\,
Q \, \longrightarrow\, 0
\end{equation}
be a short exact sequence of vector bundles on $M$. This exact sequence is called split if there
is an algebraic homomorphism $\varphi\, :\, Q\, \longrightarrow\, E$ such that $p\circ\varphi\,=\, {\rm Id}_Q$.
If \eqref{e12} is split, then $E\,=\, V\oplus Q$.

\begin{lemma}\label{lem3}
In \eqref{e12}, assume that $E$ is semiample and $Q\,=\, {\mathcal O}_M$. Then the following
two hold:
\begin{enumerate}
\item The sequence in \eqref{e12} splits.

\item $V$ in \eqref{e12} is semiample.
\end{enumerate}
\end{lemma}

\begin{proof}
Let
\begin{equation}\label{e11b}
\theta\,\,\in\,\, H^1(M,\, \text{Hom}({\mathcal O}_M,\, V))\,\,=\,\, H^1(M,\, V)
\end{equation}
be the extension class for \eqref{e12}. To prove that \eqref{e12} splits we need to show that $\theta\,=\, 0$.

Fix a positive integer $d$ such that the line bundle ${\mathcal O}_{{\mathbb P}(E)}(d)$ is generated by
its global sections. Let
\begin{equation}\label{e10}
{\mathcal K}(d)\,:=\, \text{kernel}(\text{Sym}^d(p))\, \subset\, \text{Sym}^d(E)
\end{equation}
be the kernel of the homomorphism $\text{Sym}^d(p)\, :\, \text{Sym}^d(E) \, \longrightarrow\,
\text{Sym}^d(Q)\,=\, {\mathcal O}_M$ given by $p$ in \eqref{e12}. Note that ${\mathcal K}(d)$ in \eqref{e10}
is the image of $V\otimes \text{Sym}^{d-1}(E)$ in $\text{Sym}^d(E)$ by the natural map. So
$V\otimes \text{Sym}^{d-1}(Q)\,=\, V$ is a quotient of ${\mathcal K}(d)$. Let
$$\text{Sym}^d(E)\,\, \longrightarrow\,\, {\mathcal V}$$ be the quotient of $\text{Sym}^d(E)$ by the
kernel of the above quotient map ${\mathcal K}(d)\, \longrightarrow\,
V$. Hence we have the following commutative diagram:
\begin{equation}\label{q2}
\begin{matrix}
0 & \longrightarrow & {\mathcal K}(d) & \longrightarrow & \text{Sym}^d(E)
& \xrightarrow{\,\,\,\text{Sym}^d(p)\,\,\,} & {\mathcal O}_M & \longrightarrow & 0\\
&&\,\,\, \Big\downarrow \widetilde{\beta} &&\,\,\, \Big\downarrow\beta && \Vert\\
0 & \longrightarrow & V\otimes \text{Sym}^{d-1}(Q)\,=\, V & \longrightarrow & {\mathcal V}
& \longrightarrow & {\mathcal O}_M & \longrightarrow & 0
\end{matrix}
\end{equation}
where ${\mathcal K}(d)$ is defined in \eqref{e10} and $\beta$ is the quotient map while
$\widetilde\beta$ is the restriction of $\beta$. The extension class for the bottom exact
sequence in \eqref{q2} is $d\cdot \theta$, where $\theta$ is the class in \eqref{e11b}. To see this,
take a covering $\{{\mathcal U}_i\}_{i\in I}$ of $M$ by affine open subsets and also take
$s_i\, \in\, H^0({\mathcal U}_i,\, E\big\vert_{{\mathcal U}_i})$, for all $i\, \in\, I$, such that
$p(s_i)$ is the constant function $1$ on ${\mathcal U}_i$, where $p$ is the
projection in \eqref{e12}. The cohomology class $\theta$ in \eqref{e11b}
is given by the $1$-cocycle defined by
$$
s_i-s_j\,\in\, H^0({\mathcal U}_i\cap {\mathcal U}_j,\, V\big\vert_{{\mathcal U}_i\cap {\mathcal U}_j}),\ \ \,
i,\, j\, \in\, I.
$$
We have
$$
\beta(s^{\otimes d}_i) - \beta(s^{\otimes d}_j) \,=\, \beta((s_j+(s_i-s_j))^{\otimes d}) - \beta(s^{\otimes d}_j)
$$
$$
\widetilde{\beta}((s_j+(s_i-s_j))^{\otimes d} - \beta(s^{\otimes d}_j)) \,=\, 
d\cdot p(s_j)^{\otimes (d-1)}(s_i-s_j)\,=\, d\cdot (s_i-s_j),
$$
where $\beta$ and $\widetilde\beta$ are the homomorphisms in \eqref{q2}.
{}From this it follows immediately that the extension class for the bottom exact
sequence in \eqref{q2} is $d\cdot \theta$.

Since the line bundle ${\mathcal O}_{{\mathbb P}(E)}(d)$ is generated by
its global sections, there is a section
$$
\sigma\,\, \in\,\, H^0({\mathbb P}(E),\, {\mathcal O}_{{\mathbb P}(E)}(d))
$$
such that $\text{Sym}^d(p)(\sigma)\, \in\, H^0(M\,\, {\mathcal O}_M)$ is a nonzero section,
where $\text{Sym}^d(p)$ is the projection in \eqref{q2}. This implies that the section
$\text{Sym}^d(p)(\sigma)$ is nowhere vanishing. Consequently, the homomorphism
$$
{\mathcal O}_M\, \longrightarrow\, {\mathcal V},\ \ \, f\, \longrightarrow\,f\cdot \beta(\sigma)
$$
gives a splitting of the bottom exact sequence in \eqref{q2}. Since
the extension class for the bottom exact sequence in
\eqref{q2} is $d\cdot \theta$, where $\theta$ is the class in \eqref{e11b}, we conclude that
$d\cdot\theta\,=\, 0$. Hence $\theta\,=\, 0$ and \eqref{e12} splits.

Since \eqref{e12} splits, the vector bundle $V$ in \eqref{e12} is a quotient of $E$. Now
$V$ is semiample because $E$ is so.
\end{proof}

The following is derived using Lemma \ref{lem3}.

\begin{corollary}\label{cor1}
In \eqref{e12}, assume that $E$ is semiample and $Q\,=\, {\mathcal O}^{\oplus r}_M$.
Then the sequence in \eqref{e12} splits.
\end{corollary}

\begin{proof}
Since Lemma \ref{lem3} covers the case of $r\,=\, 1$, we assume that $r\, \geq\, 2$.
Fix a filtration of subbundles of ${\mathcal O}^{\oplus r}_M$:
\begin{equation}\label{e13}
0\,=\, {\mathcal O}^{\oplus 0}_M\, \subset\, {\mathcal O}^{\oplus 1}_M\, \subset\, 
{\mathcal O}^{\oplus 2}_M\, \subset\, \cdots\, \subset\, 
{\mathcal O}^{\oplus (r-1)}_M\, \subset\, {\mathcal O}^{\oplus r}_M.
\end{equation}
For $0\, \leq\, j\, \leq\, r-1$, let
\begin{equation}\label{e14}
\Phi_j\,\, :\,\, {\mathcal O}^{\oplus r}_M\,\, \longrightarrow\,\, 
{\mathcal O}^{\oplus r}_M/{\mathcal O}^{\oplus j}_M
\end{equation}
be the quotient map, where ${\mathcal O}^{\oplus j}_M$ is the subbundle in \eqref{e13}.
The composition of homomorphisms
$$
E\,\, \stackrel{p}{\longrightarrow}\,\, Q\,\,=\,\, {\mathcal O}^{\oplus r}_M \,\,
\stackrel{\Phi_j}{\longrightarrow}\,\, {\mathcal O}^{\oplus r}_M/{\mathcal O}^{\oplus j}_M,
$$
where $p$ and $\Phi_j$ are the projections in \eqref{e12} and \eqref{e14} respectively,
will be denoted by $\Psi_j$. Let
\begin{equation}\label{e15}
V(j)\,\, :=\,\, \text{kernel}(\Psi_j) \,\, \subset\,\, E
\end{equation}
be the kernel of this composition of homomorphisms.

{}From Lemma \ref{lem3} we know that $E\,=\, V(r-1)\oplus {\mathcal O}_M$ and $V(r-1)$ is
semiample. So Lemma \ref{lem3} says that $V(r-1)\,=\, V(r-2)\oplus {\mathcal O}_M$ and $V(r-2)$ is
semiample. Now we apply Lemma \ref{lem3} to $V(r-2)$ and keep iterating. After $r$ steps we
obtain that $E\,=\, V\oplus {\mathcal O}^{\oplus r}_M$ and the sequence in \eqref{e12} splits.
\end{proof}

\begin{proposition}\label{prop2}
In \eqref{e12}, assume that $E$ is semiample and $c_1(Q)\,=\, 0$.
Then the sequence in \eqref{e12} splits.
\end{proposition}

\begin{proof}
Since $E$ is semiample and $Q$ is a quotient of $E$, it follows that $Q$ is semiample. So
Theorem \ref{thm1} says that $Q$ is \'etale trivial.

Let
$$
\varpi\, :\, Y\, \longrightarrow\, M
$$
be an \'etale Galois covering such that $\varpi^* Q$ is a trivial vector bundle. Consider
the exact sequence
\begin{equation}\label{e11}
0\, \longrightarrow\, \varpi^*V \, \longrightarrow\, \varpi^* E \, \xrightarrow{\,\,\,\varpi^*p\,\,\,}\,
\varpi^*Q \, \longrightarrow\, 0
\end{equation}
obtained by pulling back \eqref{e12} to $Y$. Note that $\varpi^* E$ is semiample because $E$ is so.
Hence Corollary \ref{cor1} says that \eqref{e11} splits. Fix a splitting homomorphism
\begin{equation}\label{e16}
\varphi\, \,:\,\, \varpi^*Q \,\, \longrightarrow\,\, \varpi^* E;
\end{equation}
so we have $(\varpi^* p)\circ\varphi\,=\, {\rm Id}_{\varpi^*Q}$.

The Galois group $\Gamma\,:=\, \text{Gal}(Y/M)$ acts on all vector bundles on $Y$ pulled back from $M$.
In particular, $\Gamma$ acts on both $\varpi^*E$ and $\varpi^*Q$. For any
$\gamma\, \in\, \Gamma$, let
$$
\varphi^\gamma \, \,:\,\, \varpi^*Q \,\, \longrightarrow\,\, \varpi^* E
$$
be the homomorphism given by the following composition of maps:
$$
\varpi^*Q \,\, \xrightarrow{\,\,\,\,\gamma\cdot\,\,\,}\,\, \varpi^* Q
\,\, \stackrel{\varphi}{\longrightarrow}\,\, \varpi^* E
\,\, \xrightarrow{\,\,\,\gamma^{-1}\cdot\,\,\,}\,\, \varpi^* E;
$$
here ``$\cdot$'' denotes the action of $\Gamma$. Note that each $\varphi^\gamma$ is a
splitting homomorphism for the exact sequence in \eqref{e11}, meaning
$\varpi^*p\circ\varphi^\gamma\,=\, {\rm Id}_{\varpi^*Q}$, where $\varpi^*p$ is the
projection in \eqref{e11}. It is evident that
$$
\widehat{\varphi}\,\,:=\, \frac{1}{\# \Gamma} \sum_{\gamma\in \Gamma}\varphi^\gamma
\, \,:\,\, \varpi^*Q \,\, \longrightarrow\,\, \varpi^* E
$$
is a $\Gamma$--equivariant splitting of \eqref{e11}. Hence $\widehat{\varphi}$ produces
a splitting of \eqref{e12}.
\end{proof}

\section{Virtually globally generated bundles}

Henceforth, $M$ will denote an irreducible smooth complex projective curve.

A vector bundle $E$ on $M$ is called \textit{virtually globally generated} if there is an
irreducible smooth complex projective curve $Y$, and a nonconstant morphism $\phi\, :\, Y\, \longrightarrow\, M$,
such that the vector bundle $\phi^*E$ is generated by its global sections. (Note that $\phi$ is allowed
to have ramifications.)

\begin{lemma}\label{lem2}
Let $E$ be a virtually globally generated vector bundle over the curve $M$. Then $E$ is semiample.
\end{lemma}

\begin{proof}
Since $E$ is virtually globally generated, it follows that
$$
E\,\, =\,\, A\oplus B,
$$
where $A$ is an ample vector bundle over $M$ and $B$ is a finite vector bundle over $M$ \cite[p.~40, Theorem 1.1]{BP}.

An ample vector bundle is clearly semiample. So $A$ is semiample. From Theorem \ref{thm1} we know that the finite
vector bundle $B$ is semiample. It is straightforward to check that the direct sum of two semiample vector bundles
is again semiample (see \cite[p.~785, Lemma 2.1]{LN} for a proof). Consequently, the vector bundle
$E\,=\, A\oplus B$ is semiample.
\end{proof}

The following is a converse of Lemma \ref{lem2}.

\begin{proposition}\label{prop1}
Let $E$ be a semiample vector bundle over $M$. Then $E$ is virtually globally generated.
\end{proposition}

\begin{proof}
Let
\begin{equation}\label{e5}
0\,=\, F_0\, \subset\, F_1\, \subset\, \cdots \, \subset\, F_{\ell-1}\, \subset\, F_\ell \,=\, E
\end{equation}
be the Harder--Narasimhan filtration of $E$ (see \cite[p.~16, Theorem 1.3.4]{HL}); so $E$ is
semistable if and only if $\ell\,=\,1$. Consider the quotient vector bundle
$Q\, :=\, E/F_{\ell-1}$ in \eqref{e5}. The quotient map
\begin{equation}\label{q}
q\, :\, E\, \longrightarrow\, E/F_{\ell-1} \,=\, Q
\end{equation}
produces an embedding
$$
\iota\,\, :\,\, {\mathbb P}(Q) \,\, \hookrightarrow\, {\mathbb P}(E),
$$
such that $\iota^* {\mathcal O}_{{\mathbb P}(E)}(1)\,=\, {\mathcal O}_{{\mathbb P}(Q)}(1)$.
Therefore, the condition that $E$ is semiample implies that $Q$ is semiample. In particular,
for some positive integer $n_0$, the symmetric product $\text{Sym}^{n_0}(Q)$ admits a
nonzero section.

The above vector bundle $\text{Sym}^{n_0}(Q)$ is semistable because $Q$ is so
\cite[p.~285, Theorem 3.18]{RR}. Now, since the semistable vector bundle $\text{Sym}^{n_0}(Q)$ admits
a nonzero section, we conclude that $$\text{degree}(\text{Sym}^{n_0}(Q))\,\, \geq\,\, 0.$$ This implies that
\begin{equation}\label{e6}
\text{degree}(Q)\,\, \geq\,\, 0.
\end{equation}

We first assume that
\begin{equation}\label{e7}
\text{degree}(Q)\,\, > \,\, 0.
\end{equation}

We observe that any semistable vector bundle $W$ on $M$ of positive degree is ample.
Indeed, the degree of any quotient of a semistable vector bundle $W$ of positive degree is
positive; this follows immediately from the definition of semistability. This
implies that $W$ is ample \cite[p.~80, Proposition 2.1(ii)]{Ha1}. From \eqref{e7} it follows immediately
that the degree of all quotients $F_i/F_{i-1}$, $1\, \leq\, i\, \leq\, \ell-1$, in \eqref{e5} are positive.
So $F_i/F_{i-1}$ is ample for every $1\, \leq\, i\, \leq\, \ell-1$.

If $0\, \longrightarrow\, A \, \longrightarrow\, B \, \longrightarrow\, C \, \longrightarrow\, 0$ is a short 
exact sequence of vector bundles on $M$ where $A$ and $C$ are ample, then $B$ is also ample \cite[p.~13, 
Proposition 6.1.13(ii)]{La2}. Since $F_i/F_{i-1}$ in \eqref{e5} is ample for every $1\, \leq\, i\, \leq\, 
\ell-1$, we now conclude that $E$ is ample. In particular, $E$ is semiample.

In view of \eqref{e6}, now assume that
\begin{equation}\label{e8}
\text{degree}(Q)\,\, = \,\, 0.
\end{equation}

First assume that $E$ is semistable. Then Theorem \ref{thm1} says that $E\,=\, Q$ is \'etale trivial. Hence $E\,=\, Q$
is virtually globally generated. So now we assume that $E$ is\, \textit{not}\, semistable. So in \eqref{e5}
we have $\ell\, \geq\,2$.

The kernel of the quotient map $q$ in \eqref{q} is $F_{\ell-1}$. We have the
short exact sequence
$$
0 \, \longrightarrow\, F_{\ell-1} \, \longrightarrow\, E \, \stackrel{q}{\longrightarrow}\, Q
\, \longrightarrow\, 0.
$$
In view of \eqref{e8}, from Proposition \ref{prop2} we conclude that $E\,=\, F_{\ell-1}\oplus Q$.

The vector bundle $F_{\ell-1}$ is ample because its Harder--Narasimhan filtration is
$$
0\,=\, F_0\, \subset\, F_1\, \subset\, \cdots \, \subset\, F_{\ell-1},
$$
and we have $\text{degree}(F_i/F_{i-1}) \, >\, 0$ for all $1\,\leq\, i\,\leq\, \ell-1$. Also,
$Q$ is \'etale trivial by Theorem \ref{thm1}.
Since $F_{\ell-1}$ is ample and $Q$ is \'etale trivial, it follows from
\cite[p.~40, Theorem 1.1]{BP} that $E\,=\, F_{\ell-1}\oplus Q$ is virtually globally generated.
\end{proof}

Lemma \ref{lem2} and Proposition \ref{prop1} together give the following:

\begin{theorem}\label{thm2}
A vector bundle $E$ on $M$ is semiample if and only if $E$ is virtually globally generated.
\end{theorem}

\section*{Acknowledgements}

We thank the referee for helpful comments. The first-named author is partially supported
by a J. C. Bose Fellowship (JBR/2023/000003).

\end{document}